\documentclass[a4paper,11pt,reqno]{amsart}
\usepackage{amsmath}
\usepackage{amsthm,enumerate}
\usepackage{mathtools}
\usepackage[T1]{fontenc}
\usepackage[utf8]{inputenc}
\usepackage{color}
\usepackage{dsfont}
\usepackage{fancyhdr}
\usepackage{calc}
\usepackage{url}

\usepackage[text={6.0in,8.6in},centering]{geometry}

\theoremstyle{plain}
\newtheorem{theorem}{Theorem}[section]
\newtheorem{proposition}[theorem]{Proposition}
\newtheorem{lemma}[theorem]{Lemma}
\newtheorem{corollary}[theorem]{Corollary}

\theoremstyle{definition}
\newtheorem{definition}[theorem]{Definition}

\theoremstyle{remark}
\newtheorem{remark}[theorem]{Remark}

\DeclareMathOperator\arctanh{arctanh}

\newcommand{\nc}{\normalcolor}
\newcommand{\mm}{\mathfrak m}
\newcommand{\R}{\mathbb R}
\newcommand{\Per}{\mathrm{Per}}

\begin{document}
\title[Sharp Cheeger-Buser type inequalities in $ \mathsf{RCD}(K,\infty)$ spaces]{Sharp Cheeger-Buser type inequalities in $ \mathsf{RCD}(K,\infty)$ spaces}
\author{Nicol\`o De Ponti}\thanks{Nicol\`o De Ponti: Dipartimento di Matematica ``Casorati'',  Universit\'a degli Studi di Pavia,  Italy. \\email: nicolo.deponti01@universitadipavia.it} \author{Andrea Mondino}\thanks{Andrea Mondino: University of Oxford, Mathematical Institute,  United Kingdom. \\ email: Andrea.Mondino@maths.ox.ac.uk}

\begin{abstract}
The goal of the paper is to sharpen and generalise bounds involving Cheeger's isoperimetric constant $h$ and the first eigenvalue $\lambda_{1}$ of the Laplacian. 
\\A celebrated lower bound of $\lambda_{1}$ in terms of $h$, $\lambda_{1}\geq h^{2}/4$, was proved by Cheeger in 1970 for smooth Riemannian manifolds. An upper bound on  $\lambda_{1}$ in terms of $h$ was established by Buser in 1982 (with dimensional constants) and improved (to a dimension-free  estimate) by Ledoux in 2004 for smooth Riemannian manifolds with Ricci curvature bounded below.
\\ The goal of the paper is twofold. First: we sharpen the inequalities obtained by Buser and Ledoux obtaining a dimension-free sharp Buser inequality for spaces with (Bakry-\'Emery weighted) Ricci  curvature bounded below by $K\in \R$ (the inequality is sharp for $K>0$ as equality is obtained on the Gaussian space). Second: all of our results hold in the higher generality of (possibly non-smooth) metric measure spaces with Ricci curvature bounded below in synthetic sense, the so-called $ \mathsf{RCD}(K,\infty)$ spaces.
\end{abstract}

\maketitle

\section{Introduction}
Throughout the paper  $(X, \mathsf{d})$ will be a complete metric space and $\mm$ will be a non-negative Borel measure on $X$, finite on bounded subsets. The triple $(X, \mathsf{d}, \mm)$ is called \emph{metric measure space}, m.m.s. for short.
We denote by $\mathsf{Lip}(X)$ the space of real-valued Lipschitz functions over $X$ and we write $f\in \mathsf{Lip}_b(X)$ if $f\in \mathsf{Lip}(X)$ and $f$ is bounded with bounded support. Given $f\in \mathsf{Lip}(X)$ its slope $|\nabla f|(x)$ at $x\in X$ is defined by
\begin{equation}
|\nabla f|(x):=\limsup_{y\rightarrow x} \frac{|f(y)-f(x)|}{\mathsf{d}(y,x)},
\end{equation}  
with the convention $|\nabla f|(x)=0$ if $x$ is an isolated point.
The first non-trivial eigenvalue of the Laplacian is characterized as follows:
\begin{itemize}
\item If $\mm(X)<\infty$, the non-zero constant functions are in $L^2(X,\mm)$ and are eigenfunctions of  the Laplacian with eigenvalue $0$. In this case,  we set 
\begin{equation}\label{eq:defla1Intro}
\lambda_1= \inf\bigg\{\frac{\int_X |\nabla f|^2 d\mm}{\int_X |f|^2d\mm}: \ 0\not\equiv f\in \mathsf{Lip}_{b}(X), \int_X f\, d\mm=0\bigg\}.
\end{equation}
\item When $\mm(X)=\infty$, $0$  may not  be an eigenvalue of the Laplacian. Thus,  we set
\begin{equation}\label{eq:defla0Intro}
\lambda_0= \inf\bigg\{\frac{\int_X |\nabla f|^2 d\mm}{\int_X |f|^2d\mm}: \  \ 0\not\equiv f\in \mathsf{Lip}_b(X) \bigg\}.
\end{equation}
\end{itemize}
At this level of generality, the spectrum of the Laplacian may not be discrete (see Remark \ref{rem:SpectrDeltaDiscrete} for more details); in any case the definitions \eqref{eq:defla1Intro} and \eqref{eq:defla0Intro} make sense, and one can investigate bounds on $\lambda_{1}$ and $\lambda_{0}$.
\\Note  that $\lambda_0$ may be zero (for instance if $\mm(X)<\infty$ or if $(X,\mathsf{d},\mm)$ is the Euclidean space $\mathbb{R}^d$ with the Lebesgue measure) but there are examples when  $\lambda_0> 0$:  for instance in the Hyperbolic plane $\lambda_{0}=1/4$ and more generally on an $n$-dimensional simply-connected Riemannian manifold with sectional curvatures bounded above by $k<0$ it holds $\lambda_{0}\geq (n-1)^{2} |k|  /4$ (see \cite{McKean}). 
\\

Given a Borel subset $A\subset X$ with $\mm(A)<\infty$, the \textit{perimeter} $\mathrm{Per}(A)$ is defined as follows (see for instance \cite{MirandaJr}):
\begin{equation*}
\mathrm{Per}(A):=\inf\bigg\{\liminf_{n\rightarrow \infty}\int_X |\nabla f_n|d\mm: f_n\in \mathsf{Lip}_b(X), f_n\rightarrow \chi_A \ \mathrm{in} \ L^1(X,\mm)\bigg\}.
\end{equation*}
In 1970, Cheeger \cite{ChIn} introduced an isoperimetric constant, now known as \emph{Cheeger constant}, to bound from below the first eigenvalue of the Laplacian.
The \textit{Cheeger constant} of the metric measure space $(X,\mathsf{d},\mm)$ is defined by 
\begin{equation}\label{eq:defChConst}
h(X):=
\begin{cases}
\inf  \left\{\frac{\mathrm{Per}(A)}{\mm(A)}\, :\, A\subset X \text{ Borel subset with $\mm(A)\leq \mm(X)/2$} \right\} & \text {if } \mm(X)<\infty \\
\inf  \left\{\frac{\mathrm{Per}(A)}{\mm(A)}\, :\, A\subset X \text{ Borel subset with $\mm(A)<\infty$} \right\} &\text {if } \mm(X)=\infty.
 \end{cases}
\end{equation}
The lower bound obtained in \cite{ChIn} for compact Riemannian manifolds, now known as  \emph{Cheeger inequality}, reads as 
\begin{equation}\label{eq:ChIn}
\lambda_{1}\geq \frac{1}{4} h(X)^{2}.
\end{equation}
As proved by Buser \cite{Buser78}, the constant $1/4$ in \eqref{eq:ChIn} is  optimal  in the following sense: for any $h > 0$ and $\varepsilon>0$, there exists a closed  (i.e. compact without boundary)  two-dimensional Riemannian manifold $(M,g)$ with $h(M)=h$ and such that $\lambda_{1}\leq  \frac 1 4 h(M)^{2}+\varepsilon$.
\\The paper  \cite{ChIn} is in the framework of smooth  Riemannian manifolds; however, the stream of arguments (with some care) extends to general metric measure spaces. For the reader's  convenience, we give a self-contained proof of \eqref{eq:ChIn} for m.m.s. in the Appendix (see Theorem \ref{thm:Cheeger}).
\\Cheeger's inequality \eqref{eq:ChIn} revealed to be extremely useful in proving lower bounds on the first eigenvalue of the Laplacian in terms of the isoperimetric constant $h$. It was thus an important discovery by Buser \cite{Buser} that also an upper bound for $\lambda_{1}$ in terms of $h$ holds, where the inequality explicitly depends on the lower bound on the Ricci curvature of the smooth Riemannian manifold. More precisely, Buser \cite{Buser}  proved that for any compact Riemannian manifold of dimension $n$ and ${\rm Ric}\geq K$, $K\leq 0$ it holds
\begin{equation}\label{eq:BuserConst}
\lambda_1\leq 2\sqrt{-(n-1)K}h+10h^2.
\end{equation}
Note that the constant here is dimension-dependent. For a complete connected Riemannian manifold with ${\rm Ric}\geq K$, $K\leq 0$, Ledoux \cite{Ledoux} remarkably  showed that the constant can be chosen to be independent of the dimension:
\begin{equation}\label{eq:LedouxConst}
\lambda_1\leq \max\{6\sqrt{-K}h,36h^2\}. 
\end{equation}

The goal of the present work is twofold: 
\begin{enumerate}
\item The main results of the paper (Theorem \ref{thm:MainImpl} and Corollary \ref{cor:explicitbounds}) improve the constants in both the Buser-type inequalities \eqref{eq:BuserConst}-\eqref{eq:LedouxConst} in a way that now the inequality is sharp for $K>0$ (as equality is attained on the Gaussian space).
\item The inequalities are established in the higher generality of (possibly non-smooth) metric measure spaces satisfying Ricci curvature lower bounds in synthetic sense, the so-called $\mathsf{RCD}(K,\infty)$ spaces.  
\end{enumerate}
For the precise definition of $\mathsf{RCD}(K,\infty)$ space, we refer the reader to Section \ref{sec:prel}. Here let us just recall that the $\mathsf{RCD}(K,\infty)$ condition was introduced by Ambrosio-Gigli-Savar\'e \cite{AGS1} (see also \cite{AGMR}) as a refinement of the $\mathsf{CD}(K,\infty)$ condition of Lott-Villani \cite{LV} and Sturm \cite{St}. Roughly, a $\mathsf{CD}(K,\infty)$ space is a (possibly infinite-dimensional, possibly non-smooth) metric measure space with Ricci curvature bounded from below by $K$, in a synthetic sense.  While the $\mathsf{CD}(K,\infty)$ condition allows Finsler structures, the main point of $\mathsf{RCD}$ is to reinforce the axiomatization (by asking linearity of the heat flow) in order to rule out Finsler structures and thus isolate the ``possibly non-smooth Riemannian structures with Ricci curvature bounded below''. It is out of the scopes of this introduction to survey the long list of achievements and results proved for $\mathsf{CD}$ and $\mathsf{RCD}$ spaces (to this aim, see the Bourbaki seminar \cite{VilB} and the recent ICM-Proceeding \cite{AmbrosioICM}). Let us just mention that a key property of both $\mathsf{CD}$ and $\mathsf{RCD}$ is the stability under measured Gromov-Hausdorff convergence (or more generally ${\mathbb D}$-convergence of Sturm \cite{St,AGS1}, or even more generally pointed measured Gromov convergence \cite{GMS}) of metric measure spaces. In particular pointed measured Gromov-Hausdorff  limits of Riemannian manifolds with Ricci bounded below, the so-called \emph{Ricci limits}, are examples of (possibly non-smooth) $\mathsf{RCD}$ spaces. Let us also recall that weighted Riemannian manifolds with Bakry-\'Emery Ricci tensor bounded below are also examples of $\mathsf{RCD}$ spaces; for instance the Gaussian space $(\R^{d}, |\cdot|, (2\pi)^{-d/2}  e^{-|x|^{2}/2} d\mathcal L^{d}(x))$, $1\leq d\in \mathbb{N}$, satisfies $\mathsf{RCD}(1,\infty)$. It is also worth recalling that if  $(X,\mathsf{d},\mm)$ is an $\mathsf{RCD}(K,\infty)$ space for some $K>0$, then $\mm(X)<\infty$; since scaling the measure by a constant does not affect the synthetic Ricci curvature lower bounds, when $K>0$, without loss of generality one can then assume $\mm(X)=1$.
\\ 

In order to state our main result, it is convenient to set
\begin{equation}\label{eq:ExplJKIntro}
J_K(t)=\begin{cases}\sqrt{\frac{2}{\pi K}}\arctan\Big(\sqrt{e^{2Kt}-1}\Big)  \ \ &\textrm{if} \ \  K>0,\\
\frac{2}{\sqrt{\pi}}\sqrt{t} \ \ &\textrm{if} \ \ K=0,\\
\sqrt{-\frac{2}{\pi K}}\arctanh{\Big(\sqrt{1-e^{2Kt}}\Big)} \ \ &\textrm{if} \ \ K<0.\end{cases} \qquad \forall t>0
\end{equation}

The aim of the paper is to prove the following theorem.

\begin{theorem}[Sharp implicit Buser-type inequality for $\mathsf{RCD}(K,\infty)$ spaces]\label{thm:MainImpl}
Let $(X,\mathsf{d},\mm)$ be an $\mathsf{RCD}(K,\infty)$ space, for some $K\in \mathbb{R}$. 
\begin{itemize}
\item In case $\mm(X)=1$, then
\begin{equation}\label{implicitlambda1Intro}
h(X)\geq \sup_{t>0}  \frac{1-e^{-\lambda_1t}}{J_K(t)}.
\end{equation}
The inequality is sharp for $K>0$, as equality is achieved for the Gaussian space $(\R^{d}, |\cdot|, (2\pi)^{-d/2}  e^{-|x|^{2}/2} d\mathcal L^{d}(x))$, $1\leq d\in \mathbb{N}$.
\item In case $\mm(X)=\infty$, then
\begin{equation}\label{implicitlambda0Intro}
h(X)\geq 2\sup_{t>0} \frac{1-e^{-\lambda_0t}}{J_K(t)}.
\end{equation}
\end{itemize}
\end{theorem}
Using the expression \eqref{eq:ExplJKIntro} of $J_{K}$, in the next corollary we obtain more explicit bounds.

\begin{corollary}[Explicit Buser inequality for $\mathsf{RCD}(K,\infty)$ spaces]\label{cor:explicitbounds} 
Let $(X,\mathsf{d},\mm)$ be an $\mathsf{RCD}(K,\infty)$ space, for some $K\in \mathbb{R}$. 
\begin{itemize}
\item Case $K>0$. If $\frac{K}{\lambda_{1}}\geq c>0$, then 
\begin{equation}\label{eq:la1h2cK>0}
\lambda_{1}\leq \frac{\pi}{2c} h(X)^{2}.
\end{equation}
The estimate is sharp, as equality is  attained on the Gaussian space \\ $(\R^{d}, |\cdot|, (2\pi)^{-d/2}  e^{-|x|^{2}/2} d\mathcal L^{d}(x))$, $1\leq d\in \mathbb{N}$,  for which $K=1, \lambda_{1}=1, h(X)= (2/\pi)^{1/2}$.
\item Case $K=0$, $\mm(X)=1$. It holds
\begin{equation}\label{eq:K0m1}
\lambda_{1}\; \leq\;  \frac{4}{\pi} h(X)^2 \inf_{T>0} \frac{T}{(1-e^{-T})^{2}} \;  < \; \pi  h(X)^2.
\end{equation}
In case $\mm(X)=\infty$, the estimate \eqref{eq:K0m1} holds replacing $\lambda_{1}$ with $\lambda_{0}$ and $h(X)$ with $h(X)/2$.

\item Case $K<0$, $\mm(X)=1$. It holds
\begin{align}
\lambda_1&\leq \max\bigg\{\sqrt{-K}\frac{\sqrt{2}\log\big(e+\sqrt{e^{2}-1}\big)}{\sqrt{\pi}(1-\frac{1}{e})}h(X),\frac{2\Big(\log\big(e+\sqrt{e^{2}-1}\big)\Big)^2}{\pi\Big(1-\frac{1}{e}\Big)^2}h(X)^2\bigg\} \nonumber \\
&<\max\left\{\frac{21}{10} \sqrt{-K}h(X),\frac{22}{5}h(X)^2 \right\}. \label{eq:K<0m1}
\end{align}
In case $\mm(X)=\infty$, the estimate \eqref{eq:K<0m1} holds replacing $\lambda_{1}$ with $\lambda_{0}$ and $h(X)$ with $h(X)/2$.
\end{itemize}
\end{corollary}

\begin{remark}\label{rem:SpectrDeltaDiscrete}
Even if the definitions of  $\lambda_{0}$ and $\lambda_{1}$ as in \eqref{eq:defla1Intro} and \eqref{eq:defla0Intro}  make sense regardless of the discreteness of the spectrum of the Laplacian (as well as the proofs of the above results), it is worth to mention some cases of interest where the Laplacian has discrete spectrum.
\\It was proved in \cite{GMS} that an $\mathsf{RCD}(K,\infty)$ space, with $K>0$ (or with finite diameter) has discrete spectrum (as the Sobolev imbedding $\mathbb{V}$ into $L^{2}$ is compact). Even in case of infinite measure the  embedding of $\mathbb{V}$ in $L^2$ may be compact. An example is given by $\mathbb{R}$ with the Euclidean distance $\mathsf{d}(x,y)=|x-y|$ and the measure $\mm:=\frac{1}{\sqrt{2\pi}}e^{x^2/2}d\mathcal{L}^1.$ It is a $\mathsf{RCD}(-1,\infty)$ space and a result of Wang \cite{Wang} ensures that the spectrum is discrete. 
\end{remark}

\subsection*{Comparison with previous results in the literature}
Theorem \ref{thm:MainImpl} and Corollary \ref{cor:explicitbounds}  improve the known results about Buser-type inequalities in several aspects. First of all the best results obtained before this paper are the aforementioned estimates \eqref{eq:BuserConst}-\eqref{eq:LedouxConst} due to Buser \cite{Buser} and Ledoux \cite{Ledoux} for smooth complete Riemannian manifolds satisfying  ${\rm Ric}\geq K$, $K\leq 0$.
Let us stress that the constants in Corollary \ref{cor:explicitbounds}  improve  the ones in both   \eqref{eq:BuserConst}-\eqref{eq:LedouxConst} and are dimension-free as well. In addition, the improvements of the present paper are: 
\begin{itemize}
\item In case $K>0$, the inequalities \eqref{implicitlambda1Intro} and \eqref{eq:la1h2cK>0}  are sharp (as equality is attained on the Gaussian space).
\item The results hold in the higher generality of (possibly non-smooth) $\mathsf{RCD}(K,\infty)$ spaces.
\end{itemize}
The proof of Theorem \ref{thm:MainImpl}   is inspired by the semi-group approach of Ledoux \cite{Ledoux0, Ledoux}, but it improves upon by using Proposition  \ref{pr: Bakry-Ledoux} in place of:
\begin{itemize}
 \item A dimension-dependent Li-Yau inequality, in \cite{Ledoux0}.
 \item A weaker version of  Proposition  \ref{pr: Bakry-Ledoux} (see  \cite[Lemma 5.1]{Ledoux}) analyzed only in case $K\leq 0$, in \cite{Ledoux}.
 \end{itemize}

Theorem \ref{thm:MainImpl} and Corollary \ref{cor:explicitbounds} are also the first \emph{upper bounds} in the literature of $\mathsf{RCD}$ spaces  for the first eigenvalue of the Laplacian. On the other hand, \emph{lower bounds} on the first  eigenvalue of the Laplacian have been throughly analyzed in both  $\mathsf{CD}$  and $\mathsf{RCD}$ spaces:  the sharp Lichnerowitz spectral gap $\lambda_{1}\geq KN/(N-1)$ was proved under the (non-branching) $\mathsf{CD}(K,N)$ condition by Lott-Villani \cite{LV1}, under the $\mathsf{RCD}^{*}(K,N)$ condition by Erbar-Kuwada-Sturm \cite{EKS},   and generalized by Cavalletti and Mondino \cite{CaMoGT}  to a sharp spectral gap for the $p$-Laplacian for essentially non-branching  $\mathsf{CD}^{*}(K,N)$ spaces involving also an upper bound on the diameter (together with rigidity and almost rigidity statements).  Jiang-Zhang \cite{JZ1} independently showed, for $p=2$, that the improved version under an upper diameter bound holds for $\mathsf{RCD}^*(K,N)$. The rigidity of the Lichnerowitz spectral gap for   $\mathsf{RCD}^*(K,N)$ spaces, $K>0$, $N\in (1,\infty)$, known as Obata's Theorem was first proved by Ketterer \cite{Ket}. The rigidity in the Lichnerowitz spectral gap for $\mathsf{RCD}(K,\infty)$ spaces, $K>0$, was recently proved by Gigli-Ketterer-Kuwada-Ohta \cite{GKKO}. Local Poincar\'e inequalities in  the framework of $\mathsf{CD}(K,N)$ and $\mathsf{CD}(K,\infty)$ spaces were proved by Rajala \cite{RaCV}.
Finally various lower bounds, together with rigidity and almost rigidity statements for the \emph{Dirichlet first eigenvalue} of the Laplacian, have been proved by Mondino-Semola \cite{MoSe} in the framework of $\mathsf{CD}$  and $\mathsf{RCD}$ spaces. 
Lower bounds on Cheeger's isoperimetric constant have been obtained for  (essentially non-branching) $\mathsf{CD}^*(K,N)$ spaces by Cavalletti-Mondino \cite{CaMoInv, CaMoGT, CaMoLinc} and for  $\mathsf{RCD}(K,\infty)$ spaces ($K>0$) by Ambrosio-Mondino \cite{AM}.  The local and global stability properties of eigenvalues and eigenfunctions in the framework of $\mathsf{RCD}$ spaces have been investigated by Gigli-Mondino-Savar\'e in [18] and by Ambrosio-Honda in \cite{AmbHondaLocal, AmbHondaLp}.

\section*{Acknowledgements}
The work has been developed when N. DP. was visiting the Mathematics Institute at the University of Warwick during fall term 2018. He wishes to thank the Institute  for the excellent working conditions and the stimulating atmosphere. 
\\ N.DP. is supported by the GNAMPA Project 2019 ``Trasporto ottimo per dinamiche con interazione''.
\\ A.M. is supported by the EPSRC First Grant EP/R004730/1  ``Optimal transport and Geometric Analysis'' and by the ERC Starting Grant  802689 ``CURVATURE''.

\nc

\section{Preliminaries}\label{sec:prel}
Throughout the paper, unless otherwise stated, we assume $(X,\mathsf{d})$ is a complete and separable metric space. We endow $(X,\mathsf{d})$ with a reference $\sigma$-finite non-negative measure $\mm$ over the Borel $\sigma$-algebra $\mathcal{B}$, with  $\textsf{supp}(\mm)=X$ and satisfying an exponential growth condition: namely that there exist $x_0\in X$, $M>0$ and $c\geq 0$ such that
$$\mm(B_r(x_0))\leq M\exp(cr^2) \ \ \textrm{for\ every} \ r\geq 0.$$
Possibly enlarging $\mathcal{B}$ and extending $\mm$, we assume that $\mathcal{B}$ is $\mm$-complete. 
 The triple $(X,\mathsf{d},\mm)$  is called metric measure space, m.m.s for short.
\\

\noindent 
We denote by $\mathcal{P}_2(X)$ the space of probability measures on $X$ with finite second moment and we endow this space  with the Kantorovich-Wasserstein distance $W_2$ defined as follows:  for $\mu_0,\mu_1 \in \mathcal{P}_{2}(X)$ we set
\begin{equation}\label{eq:Wdef}
  W_2^2(\mu_0,\mu_1) := \inf_{ \pi} \int_{X\times X} \mathsf{d}^2(x,y) \, d\pi,
\end{equation}
where the infimum is taken over all $\pi \in \mathcal{P}(X \times X)$ with $\mu_0$ and $\mu_1$ as the first and the second marginal.
\\The \textit{relative entropy functional} $\mathsf{Ent}_{\mm}:\mathcal{P}_2(X)\rightarrow \R\cup \{\infty\}$ is defined as
\begin{equation}
\mathsf{Ent}_{\mm}(\mu):=\begin{cases} \int \rho\log \rho \, d\mm \ &\textrm{if} \ \mu=\rho\mm, \\ \infty \ &\textrm{otherwise}.\end{cases}
\end{equation} 
A curve $\gamma:[0,1]\rightarrow X$ is a \textit{geodesic} if 
\begin{equation}
\mathsf{d}(\gamma_s,\gamma_t)=  |t-s|\, \mathsf{d}(\gamma_0,\gamma_1) \ \ \ \ \ \forall s,t\in [0,1].
\end{equation}
In the sequel we use the notation:
$$
D(\mathsf{Ent}_{\mm}):=\{\mu\in \mathcal{P}_2(X) \,:\, \mathsf{Ent}_{\mm}(\mu)\in \R\}.
$$
We now define the $\mathsf{CD}(K,\infty)$ condition, coming from the seminal works of Lott-Villani \cite{LV} and Sturm \cite{St}.
\begin{definition}[$\mathsf{CD}(K,\infty)$ condition]
Let $K\in \mathbb{R}$. We say that  $(X,\mathsf{d},\mm)$ is a $\mathsf{CD}(K,\infty)$ space provided that for any $\mu^0,\mu^1\in D(\mathsf{Ent}_{\mm})$ there exists a $W_2$-geodesic $(\mu_t)$ such that $\mu_0=\mu^0$, $\mu_1=\mu^1$ and 
\begin{equation}
\mathsf{Ent}_{\mm}(\mu_t)\leq (1-t)\mathsf{Ent}_{\mm}(\mu_0)+t\mathsf{Ent}_{\mm}(\mu_1)-\frac{K}{2}t(1-t)W_2^2(\mu_0,\mu_1).
\end{equation}
\end{definition}

The space of continuous function $f:X\rightarrow \R$ is denoted by $\mathcal{C}(X)$ and the  Lebesgue space  by $L^p(X,\mm)$, $1\leq p\leq \infty$.

The \textit{Cheeger energy} (introduced in \cite{Ch} and further studied in \cite{AGS}) is defined as the $L^{2}$-lower semicontinuous envelope of the functional $f \mapsto \frac{1}{2}\int_X |\nabla f|^2d\mm$, i.e.:
\begin{equation}\label{eq:defChm}
\mathsf{Ch}_{\mm}(f):=\inf\bigg\{\liminf_{n\to \infty}\frac{1}{2}\int_X |\nabla f_n|^2d\mm: f_n\in \mathsf{Lip}_b(X), f_n\rightarrow f \ \mathrm{in} \ L^2(X,\mm)\bigg\}.
\end{equation}
If $\mathsf{Ch}_{\mm}(f)<\infty$, it was proved in \cite{Ch, AGS} that the set
$$
{\mathrm G}(f):=\{ g\in L^{2}(X,\mm)\,:\, \exists (f_{n})_n\subset  \mathsf{Lip}_{b}(X), \, f_{n}\to f, |\nabla f_{n}| \rightharpoonup h\leq  g \text{ in } L^{2}(X,\mm)\}
$$
is closed and convex, therefore it admits a unique element of minimal norm called \emph{minimal weak upper gradient} and denoted by $|Df|_{w}.$ The Cheeger energy can be then represented by integration as
$$
\mathsf{Ch}_{\mm}(f)=  \frac{1}{2}\int_X |D f|_{w}^2d\mm.
$$

We recall that the minimal weak upper gradient satisfies the following property (see e.g. \cite[equation (2.18)]{AGS1}):
\begin{equation}\label{locality}
|D f|_{w}=0 \ \ \mm \text{-a.e. \ on \ the \ set} \{f=0\}.
\end{equation}
One can show that $\mathsf{Ch}_{\mm}$ is a $2$-homogeneous, lower semicontinuous, convex functional on $L^2(X,\mm)$ whose proper domain
$$
\mathbb{V}:=\{f\in L^{2}(X,\mm)\,:\, \mathsf{Ch}_{\mm}(f)<\infty\}
$$
 is a dense linear subspace of $L^2(X,\mm)$. It then admits an $L^{2}$ gradient flow which is a continuous semi-group of contractions $(H_{t})_{t\geq 0}$ in $L^{2}(X,\mm)$, whose continuous trajectories $t \mapsto H_{t} f$, for $f\in L^{2}(X,\mm)$, are locally Lipschitz curves from $(0,\infty)$ with values into $L^{2}(X,\mm)$ that satisfy
\begin{equation}
\frac{d}{d t} H_{t} f \in -\partial\mathsf{Ch}_{\mm}(H_t f) \ \ \textrm{for a.e.} \ t\in (0,\infty).
\end{equation}  
Here $\partial$ denotes the subdifferential of convex analysis, namely for every $f\in \mathbb{V}$ we have $\ell \in \partial\mathsf{Ch}_{\mm}(f)$ if and only if
\begin{equation}
\int_X \ell (g - f)d \mm \leq \mathsf{Ch}_{\mm}(g) - \mathsf{Ch}_{\mm}(f) ,\ \ \ \textrm{for every} \ g\in L^2(X,\mm).
\end{equation}
 We now define the $\mathsf{RCD}(K,\infty)$ condition, introduced and throughly analyzed in \cite{AGS1} (see also \cite{AGMR} for the present simplified axiomatization and the extension to the $\sigma$-finite case).
 
 \begin{definition}[$\mathsf{RCD}(K,\infty)$ condition]
Let $K\in \mathbb{R}$. We say that the metric measure space $(X,\mathsf{d},\mm)$ is  $\mathsf{RCD}(K,\infty)$ if it  satisfies the $\mathsf{CD}(K,\infty)$ condition and moreover the Cheeger energy $\mathsf{Ch}_{\mm}$ is quadratic, i.e. it satisfies the parallelogram identity
\begin{equation}
\mathsf{Ch}_{\mm}(f+g)+\mathsf{Ch}_{\mm}(f-g)=2\mathsf{Ch}_{\mm}(f)+2\mathsf{Ch}_{\mm}(g), \quad \forall f,g\in \mathbb{V}.
\end{equation}
\end{definition}
If  $(X,\mathsf{d},\mm)$ is  an $\mathsf{RCD}(K,\infty)$ space, then the Cheeger energy induces the Dirichlet form   $\mathcal{E}(f):=2\mathsf{Ch}_{\mm}(f)$ which is  strongly local, symmetric and admits the Carr\'e du Champ 
$$
\Gamma(f):= |Df|_{w}^{2}, \quad \forall f\in \mathbb{V}.
$$
The  space  $\mathbb{V}$ endowed with the norm $\left \Vert f \right\Vert^2_{\mathbb{V}}:=\left \Vert f \right\Vert^2_{L^2}+\mathcal{E}(f)$ is  Hilbert. 
Moreover, the sub-differential $\partial  \mathsf{Ch}_{m}$ is single-valued and coincides with the linear generator $-\Delta$ of the heat flow semi-group  $(H_t)_{t\geq 0}$ defined above. In other terms,
the semigroup can be equivalently characterized by the fact that for any $f\in L^2(X,\mm)$ the curve $t\mapsto H_tf\in L^2(X,\mm)$ is locally Lipschitz from $(0,\infty)$ to $L^2(X,\mm)$ and satisfies
\begin{equation}\label{eq: calore}
\begin{cases}\frac{d}{dt} H_tf=\Delta H_tf \ \ \textrm{for} \ \mathcal{L}^1\text{-a.e } t\in (0,\infty), \\
\lim_{t\to 0}H_tf=f,
\end{cases}
\end{equation}
where the limit is  in the strong  $L^2(X,\mm)$-topology.
\\The semigroup $H_t$ extends uniquely to a strongly continuous semigroup of linear contractions in $L^p(X,\mm), p\in [1,\infty)$, for which we retain the same notation. 
Regarding the case $p=\infty$, it was proved in \cite[Theorem 6.1]{AGS1} that there exists a version of the semigroup such that $H_tf(x)$ belongs to $\mathcal{C}\cap L^{\infty}((0,\infty)\times X)$ whenever $f\in L^{\infty}(X,\mm).$ We will implicitly refer to this version of $H_tf$ when $f$ is essentially bounded.
Moreover, for any $f\in L^2\cap L^{\infty}(X,\mm)$ and for every $t>0$ we have $H_tf\in \mathbb{V}\cap\mathsf{Lip}(X)$ with the explicit bound (see \cite[Theorem 6.5]{AGS1} for a proof) 
\begin{equation}\label{strong Feller}
\left\Vert |D H_tf|_{w} \right\Vert_{\infty}\leq \sqrt{\frac{K}{e^{2Kt}-1}} \left\Vert f\right\Vert_{\infty}.
\end{equation}
Two crucial properties of the heat flow are the preservation of mass and the maximum principle  (see \cite{AGS}):
\begin{align}
 \int_XH_tf \,  d\mm =\int_Xf \, d\mm, \quad \text{for any }  f\in L^1(X,\mm), \label{mass preserving} \\
 0\leq H_tf\leq C, \quad  \text{for\ any } 0\leq f\leq C\;  \mm\text{-a.e.}, \ C>0. \label{dis: principio del massimo}
\end{align}
A result of Savar\'e \cite[Corollary 3.5]{Savare}  ensures that, in the $\mathsf{RCD}(K,\infty)$ setting, for every $f\in \mathbb{V}$ and $\alpha\in [\frac{1}{2},1]$ we have
\begin{equation}
|D H_tf|_{w}^{2\alpha}\leq e^{-2\alpha Kt}H_t\big(|D f|_{w}^{2\alpha}\big), \quad \mm\text{-a.e. }.
\end{equation}
In particular, 
\begin{equation}\label{dis: Savare forte}
|D H_tf|_{w}\leq e^{-Kt}H_t(|D f|_{w}), \quad \mm\text{-a.e. }.
\end{equation}


\section{Proof of Theorem \ref{thm:MainImpl}}               
We denote by $I:[0,1]\rightarrow [0,\frac{1}{\sqrt{2\pi}}]$ the Gaussian isoperimetric function defined by $I:=\varphi\circ \Phi^{-1}$ where
$$\Phi(x):=\frac{1}{\sqrt{2\pi}}\int_{-\infty}^x e^{-u^2/2}\, du, \ \ x\in \mathbb{R},$$
and $\varphi=\Phi'$. The function $I$ is concave, continuous, $I(0)=I(1):=0$ and $0\leq I(x) \leq I(\frac{1}{2})=\frac{1}{\sqrt{2\pi}},$ for all   $x\in[0,1]$. Moreover, $I\in \mathcal{C}^{\infty}((0,1))$, it satisfies the identity
\begin{equation}\label{id: derivate I}
I(x)I''(x)=-1, \quad \text{for every } x\in (0,1).
\end{equation}
and (see \cite{BL})
\begin{equation}\label{id: asintotico I}
\lim_{x \to 0} \frac{I(x)}{x\sqrt{2\log{\frac{1}{x}}}}=1.
\end{equation}
Given $K\in \mathbb{R}$, we define the function $j_K:(0,\infty)\rightarrow (0,\infty)$ as
\begin{equation}\label{eq:defjkt}
j_K(t):=\begin{cases}\frac{K}{e^{2Kt}-1} \ \ &\mathrm{if} \ K\neq 0, \\
\frac{1}{2t} &\mathrm{if} \ K=0. \end{cases}
\end{equation}
Notice that $j_K$ is increasing as a function of $K$.
\\The next proposition was proved in the smooth setting by Bakry, Gentil and Ledoux (see \cite{BL}, \cite{BGL1} and \cite[Proposition 8.6.1]{BGL2}). 

\begin{proposition}[Bakry-Gentil-Ledoux Inequality in $\mathsf{RCD}(K,\infty)$ spaces]\label{pr: Bakry-Ledoux}
Let $(X,\mathsf{d},\mm)$ be an $\mathsf{RCD}(K,\infty)$ space, for some $K\in \mathbb{R}$. Then for every function $f\in L^2(X,\mm)$, $f:X\rightarrow [0,1]$ it holds
\begin{equation}\label{dis: Bakry-Ledoux}
|D(H_tf)|_{w}^2\leq j_K(t)\Big(\big[I(H_tf)\big]^2-\big[H_t(I(f))\big]^2\Big), \quad \mm\text{-a.e.},\; \text{for every } t>0.
\end{equation}
In particular,  for every $f\in L^2\cap L^{\infty}(X,\mm)$, it holds
\begin{equation}\label{dis: gradiente-f infinito}
\left\Vert|D(H_tf)|_{w}\right\Vert_{\infty}\leq \sqrt{\frac{2}{\pi}}\sqrt{j_K(t)}\left\Vert f \right\Vert_{\infty}, \quad \mm\text{-a.e.},\; \text{for every } t>0.
\end{equation}
\end{proposition}
\begin{proof}
Given $\varepsilon>0$, $\eta>2\varepsilon$ and $\delta>0$ sufficiently small, consider $f\in L^2(X,\mm)$ with values in $[0,1-\eta]$. We define 
\begin{align}
&\phi_{\varepsilon}(x):=I(x+\varepsilon)-I(\varepsilon),\\
&\Psi_{\varepsilon}(s):=\Big[H_s(\phi_{\varepsilon}(H_{t-s}f))\Big]^2, \quad \text{for every } s\in (0,t).
\end{align}
We notice that $\phi_{\varepsilon}(0)=0$ and $\phi_{\varepsilon}(x)\geq 0$ for every $x\in [0,1-\eta]$. Moreover, using the property \eqref{dis: principio del massimo}, $\phi_{\varepsilon}$ is Lipschitz in the range of $H_{t-s}f$. Since $t\mapsto H_{t}f$ is a locally Lipschitz map with values in $L^p(X,\mm)$ for $1<p<\infty$ (\cite[Theorem 1, Section III]{Stein}), we have that $\Psi_{\varepsilon}$ is a locally Lipschitz map with values in $L^1(X,\mm)$. 
Let  $\psi\in L^1\cap L^{\infty}(X,\mm)$ be a non-negative function.
By the chain rule for locally Lipschitz maps, the fundamental theorem of calculus for the Bochner integral and the properties of the semigroup $H_t$ we have that for any $\varepsilon>0$ and $0<\delta<t$ it holds
\begin{multline}\label{id: derivata interpolazione}
\int_X\bigg(\Big[H_{\delta}(\phi_{\varepsilon}(H_{t-\delta}f))\Big]^2-\Big[H_{t-\delta}(\phi_{\varepsilon}(H_{\delta}f))\Big]^2\bigg)\psi \, d\mm\\
=\int_{\delta}^{t-\delta}\bigg(-\frac{d}{ds}\int_X\Big[H_s(\phi_{\varepsilon}(H_{t-s}f))\Big]^2\psi \, d\mm\bigg) ds\\
=-2\int_{\delta}^{t-\delta}\bigg(\int_XH_s\big(\phi_{\varepsilon}(H_{t-s}f)\big)H_s\big(\Delta\phi_{\varepsilon}(H_{t-s}f)-\phi_{\varepsilon}'(H_{t-s}f)\Delta H_{t-s}f\big)\psi\,  d\mm\bigg)ds\\
=2\int_{\delta}^{t-\delta}\bigg(\int_X H_s\big(\phi_{\varepsilon}(H_{t-s}f)\big)H_s\big(-\phi_{\varepsilon}''(H_{t-s}f)|D H_{t-s}f|_{w}^2\big)\psi \,  d\mm\bigg)ds.
\end{multline}
Applying the Cauchy-Schwarz inequality 
\begin{equation*}
H_s(X)H_s(Y)\geq \big[H_s\big(\sqrt{XY}\big)\big]^2,
\end{equation*}
and the identity $I(x)I''(x)=-1$, for all $x\in (0,1)$, we get that the right-hand side of \eqref{id: derivata interpolazione} is bounded below by  
 \begin{equation}\label{integranda in epsilon}
2\int_{\delta}^{t-\delta}\Bigg(\int_X \Bigg[H_s\Bigg(\sqrt{\bigg(1-\frac{I(\varepsilon)}{I(H_{t-s}f+\varepsilon)}\bigg)|D H_{t-s}f|_{w}^2}\Bigg)\Bigg]^2\psi \, d\mm\Bigg)ds.
\end{equation}
Noticing  that 
$$
\int_X \Bigg[H_s\Bigg(\sqrt{\bigg(1-\frac{I(\varepsilon)}{I(H_{t-s}f+\varepsilon)}\bigg)|D H_{t-s}f|_{w}^2}\Bigg)\Bigg]^2\psi \, d\mm   \leq \int_X \Big[H_s\big(|D H_{t-s}f|_{w}\big)\Big]^2\psi \, d\mm$$
and that, for any fixed $\delta>0$,  
\begin{equation*}
\int_{\delta}^{t-\delta}\bigg(\int_X \Big[H_s\big(|D H_{t-s}f|_{w}\big)\Big]^2\psi \, d\mm\bigg)ds <\infty
\end{equation*}
thanks to the bound \eqref{strong Feller}, we can pass to the limit as $\varepsilon\to 0$ in \eqref{integranda in epsilon} using Dominated Convergence Theorem.

Since $I$ is continuous, $I(0)=0$ and $I(x)>0$ for every $x\in (0,1)$,  using the locality property \eqref{locality},
the Dominated Convergence Theorem yields
\begin{equation}\label{dis: epsilon a 0}
\int_X\bigg(\Big[H_{\delta}(I(H_{t-\delta}f))\Big]^2-\Big[H_{t-\delta}(I(H_{\delta}f))\Big]^2\bigg)\psi d\mm\geq 2\int_{\delta}^{t-\delta}\bigg(\int_X \Big[H_s\big(|D H_{t-s}f|_{w}\big)\Big]^2\psi d\mm\bigg)ds,
\end{equation}
for every $\delta\in (0,t)$. Now, we can bound the right-hand side of \eqref{dis: epsilon a 0} using the inequality \eqref{dis: Savare forte} in order to obtain
\begin{equation}
2\int_{\delta}^{t-\delta}\bigg(\int_X \Big[H_s\big(|D H_{t-s}f|_{w}\big)\Big]^2\psi d\mm\bigg)ds \geq 2\int_X \bigg(\int_{\delta}^{t-\delta}e^{2Ks}ds\bigg)|D H_tf|_{w}^2\psi \, d\mm.
\end{equation}
From \eqref{id: asintotico I} it follows that for every $0<a<1$ there exists $C=C(a)>0$ and $\bar{x}=\bar{x}(a)\in (0,1)$ such that $I(x)\leq C x^a$ for all $x\in (0, \bar{x})$. In particular, if $g\in L^2(X,\mm)$, $g:X\rightarrow [0,1-\eta]$, then $I(g)\in L^p(X,\mm)$ for every $p>2$. We now apply this argument for $p=4$, so that we can take advantage of the continuity of $I$ and the continuity of the semigroup and pass to the limit as $\delta\downarrow 0$. We obtain
\begin{equation}\label{eq:estf1-eta}
\int_X\bigg(\Big[I(H_{t}f)\Big]^2-\Big[H_{t}(I(f))\Big]^2\bigg)\psi \, d\mm\geq \frac{1}{j_K(t)}\int_X |D H_tf|_{w}^2\psi \, d\mm, 
\end{equation}
for every  $\eta>0$ sufficiently small, every $f\in L^2(X,\mm)$, $f:X\rightarrow [0,1-\eta]$.   
\\ Now, for $f\in L^2(X,\mm)$, $f:X\rightarrow [0,1]$, consider the truncation $f_{\eta}:=\min\{f,1- \eta\}$. 
Applying \eqref{eq:estf1-eta} to $f_{\eta}$, we have
\begin{equation}\label{eq:feta}
\int_X\bigg(\Big[I(H_{t}f_{\eta})\Big]^2-\Big[H_{t}(I(f_{\eta}))\Big]^2\bigg)\psi \, d\mm\geq \frac{1}{j_K(t)}\int_X |D H_tf_{\eta}|_{w}^2\psi \, d\mm. 
\end{equation}
From $f_{\eta}\to f$ in $L^{2}\cap L^{\infty}(X,\mm)$ as $\eta\downarrow 0$, we get that  $H_tf_{\eta}\to H_tf$ in $\mathbb{V}$ for every $t>0$; we can then pass to the limit as $\eta\downarrow 0$ in \eqref{eq:feta} and obtain
\begin{equation*}
\int_X\bigg(\Big[I(H_{t}f)\Big]^2-\Big[H_{t}(I(f))\Big]^2\bigg)\psi \, d\mm\geq \frac{1}{j_K(t)}\int_X |D H_tf|_{w}^2\psi \, d\mm. 
\end{equation*}
Since $\psi\in L^1\cap L^{\infty}(X,\mm)$, $\psi\geq 0$ is arbitrary,  the desired estimate \eqref{dis: Bakry-Ledoux} follows. 
\\Recalling that $0\leq I \leq \frac{1}{\sqrt{2 \pi}}$, the inequality \eqref{dis: Bakry-Ledoux} yields
\begin{equation}\label{eq:nablaHtfjK}
|D(H_tf)|_{w}\leq \sqrt{\frac{j_K(t)}{2\pi}},  \quad \mm\text{-a.e.}, \; \text{ for every  } t>0,
\end{equation}
for any $f\in L^2(X,\mm)$, $f:X\rightarrow [0,1]$. 
For any $f\in L^2\cap L^{\infty}(X,\mm)$, write $f=f^{+}-f^{-}$ with $f^{+}=\max\{f, 0\}$, $f^{-}=\max\{-f, 0\}$. Applying \eqref{eq:nablaHtfjK} to $f^{+}/\|f\|_{\infty}, f^{-}/\|f\|_{\infty}$ and summing up we obtain
\begin{align*}
\left\Vert|D H_tf|_{w}\right\Vert_{\infty} & \leq   \left\Vert|D H_tf^{+}|_{w}\right\Vert_{\infty}  +  \left\Vert|D H_tf^{-}|_{w}\right\Vert_{\infty}   \leq \sqrt{\frac{2}{\pi}}\sqrt{j_K(t)}\left\Vert f \right\Vert_{\infty}, \quad \mm\text{-a.e., }\forall t>0.
\end{align*}
\end{proof}

We next recall the definition of the first non-trivial eigenvalue of the laplacian $-\Delta$. First of all, if $\mm(X)<\infty$, the non-zero constant functions are in $L^2(X,\mm)$ and are eigenfunctions of $-\Delta$ with eigenvalue $0$. In this case,  the first non-trivial eigenvalue is given by $\lambda_{1}$
\begin{equation}\label{eq:defla1}
\lambda_1= \inf\bigg\{\frac{\int_X |D f|_{w}^2 d\mm}{\int_X |f|^2d\mm}: \ 0\not\equiv f\in \mathbb{V}, \int_X fd\mm=0\bigg\}.
\end{equation}
When $\mm(X)=\infty$, $0$  may not  be \nc an eigenvalue of $-\Delta$ and the first eigenvalue is characterized by
\begin{equation}\label{caratterizzazione lambda_0}
\lambda_0= \inf\bigg\{\frac{\int_X |D f|_{w}^2 d\mm}{\int_X |f|^2d\mm}: \  \ 0\not\equiv f\in \mathbb{V} \bigg\}.
\end{equation}

Observe that, by the very definition of Cheeger energy \eqref{eq:defChm}, the definition \eqref{eq:defla1Intro} of $\lambda_{1}$ (resp. \eqref{eq:defla0Intro} of $\lambda_{0}$) given in the Introduction in terms of slope of Lipschitz functions, is equivalent to   \eqref{eq:defla1} (resp. \eqref{caratterizzazione lambda_0}).

It is also convenient to set
\begin{equation}\label{eq:defJKt}
J_K(t):=\sqrt{\frac{2}{\pi}}\int_0^t \sqrt{j_K(s)}\, ds,
\end{equation}
where $j_{K}$ was defined in  \eqref{eq:defjkt}.

\begin{proof}[Proof of Theorem $1.1$]
\textbf{Step 1}:  Proof of  \eqref{implicitlambda1Intro}, the case  $\mm(X)=1$.  
\\First of all, we claim that for any $f\in L^2(X,\mm)$ with zero mean it holds 
\begin{equation}\label{dis: norma L2-lambda1}
\left\Vert H_tf\right\Vert_{2}\leq e^{-\lambda_1t}\left\Vert f\right\Vert_2.
\end{equation} 
To prove \eqref{dis: norma L2-lambda1} let $0\not\equiv f\in L^{2}(X, \mm)$ such that $0=\int_X fd\mm=\int_X H_tfd\mm$. Then 
\begin{equation}\label{eq:Gronwall}
2\lambda_1\int_X |H_tf|^2d\mm \leq 2\int_X |D(H_tf)|_{w}^2d\mm=-2\int_X H_tf\Delta(H_tf)d\mm=-\frac{d}{dt}\int_X |H_tf|^2d\mm,
\end{equation}
and the Gronwall's inequality yields \eqref{dis: norma L2-lambda1}.
\\Next we claim that, by duality, the bound \eqref{dis: gradiente-f infinito} implies
\begin{equation}\label{dis: f meno semigruppo-grad 1}
\left\Vert f-H_t f \right\Vert_{1}\leq J_K(t)\left\Vert |D f|_{w} \right\Vert_{1}, \quad \text{for all } f\in \mathsf{Lip}_b(X),
\end{equation}
where $J_{K}(t)$ was defined in \eqref{eq:defJKt}. 
\\To prove \eqref{dis: f meno semigruppo-grad 1} we take a function $g$, $\left\Vert g \right\Vert_{\infty} \leq 1$, and observe that
\begin{multline}\nonumber
\int_X g(f-H_t f)d\mm=-\int_{0}^t\Big(\int_X g\Delta H_s f d\mm\Big)ds=\int_{0}^t \Big(\int_X D H_s g \cdot D f d\mm\Big)ds \\
\leq \left\Vert |D f|_{w} \right\Vert_{1}\int_{0}^{t} \left\Vert|D(H_sg)|_{w}\right\Vert_{\infty}ds.
\end{multline}
Since $g$ is arbitrary,  the claimed \eqref{dis: f meno semigruppo-grad 1} follows from the last estimate combined with \eqref{dis: gradiente-f infinito}. 
\\ We now combine the above claims in order to conclude the proof.
Let $A\subset X$ be a Borel subset and let $f_n\in \mathsf{Lip}_b(X)$, $f_n\rightarrow \chi_A$ in $L^1(X,\mm)$,  be a recovery sequence for the perimeter of the set $A$, i.e.: 
$$
\mathrm{Per}(A)= \lim_{n\to \infty} \int_{X} |\nabla f_{n}|\, d\mm \geq  \limsup_{n\to \infty} \int_{X} |D f_{n}|_{w}\, d\mm.
$$
Inequality \eqref{dis: f meno semigruppo-grad 1} passes to the limit since $H_t$ is continuous in $L^1(X,\mm)$ \cite[Theorem 4.16]{AGS} and we can write
\begin{multline}\label{dis: perimetro-differenza misura e semigruppo}
J_K(t)\mathrm{Per}(A)\geq \left\Vert \chi_A-H_t(\chi_A)\right\Vert_1=\int_A [1-H_t(\chi_A)]d\mm+\int_{A^c}H_t(\chi_A)d\mm \\
=2\Big(\mm(A)-\int_A H_t(\chi_A)d\mm\Big) = 2\Big(\mm(A)-\int_X \chi_A  H_{t/2}( H_{t/2} (\chi_A))d\mm\Big)\\
= 2\Big(\mm(A)-\int_X H_{t/2}(\chi_A)  H_{t/2} (\chi_A) d\mm\Big) = 2\big(\mm(A)-\left\Vert H_{t/2}(\chi_A)\right\Vert^2_{2}\big),
\end{multline}
where we used properties \eqref{mass preserving}, \eqref{dis: principio del massimo}, together with the semigroup property and the self-adjointness of the semigroup.
We observe that $\int_{X}  H_{t/2}(\chi_A-\mm(A)) \, d\mm=0$ thanks to \eqref{mass preserving} and the fact that $H_t \mathds{1}=\mathds{1}$ when $\mm(X)=1$. We can thus apply  \eqref{dis: norma L2-lambda1} in order to bound $\left\Vert H_{t/2}(\chi_A)\right\Vert^2_{2}$ in the following way
\begin{equation}\label{dis: semigruppo-funzione caratteristica-2}
\left\Vert H_{t/2}(\chi_A)\right\Vert^2_{2}=\mm(A)^2+\left\Vert H_{t/2}(\chi_A-\mm(A))\right\Vert^2_{2}\leq \mm(A)^2+e^{-\lambda_1 t}\left\Vert \chi_A-\mm(A)\right\Vert^2_{2}.
\end{equation}
A direct computation gives $\left\Vert \chi_A-\mm(A)\right\Vert^2_{2}=\mm(A)(1-\mm(A))$,
so that  the combination of \eqref{dis: perimetro-differenza misura e semigruppo}  and  \eqref{dis: semigruppo-funzione caratteristica-2} yields
\begin{equation}\label{dis: perimetro-misura}
J_K(t)\mathrm{Per}(A)\geq 2\mm(A)(1-\mm(A))(1-e^{-\lambda_1t}), \quad \text{ for every } t>0.
\end{equation}
Recalling that in the definition of the Cheeger constant $h(X)$  one considers only Borel subsets $A\subset X$ with  $\mm(A)\leq 1/2$,
the last inequality \eqref{dis: perimetro-misura}  gives \eqref{implicitlambda1Intro}.
\\

\textbf{Step 2}: Proof of  \eqref{implicitlambda0Intro}, the case  $\mm(X)=\infty$.  
\\Arguing as in  \eqref{eq:Gronwall} using Gronwall Lemma, for any $f\in L^2(X,\mm)$ it holds
\begin{equation}\label{dis: norma L2-lambda0 misura infinita}
\left\Vert H_tf\right\Vert_{2}\leq e^{-\lambda_0 t}\left\Vert f\right\Vert_2.
\end{equation} 
Note that in order to establish \eqref{dis: perimetro-differenza misura e semigruppo},  the finiteness of $\mm(X)$ played no role. Now we can directly use \eqref{dis: norma L2-lambda0 misura infinita} to bound the right-hand side of the equation \eqref{dis: perimetro-differenza misura e semigruppo} in order to achieve 
\begin{equation}\label{dis: cheeger-autovalore lambda0}\nonumber
\frac{\mathrm{Per}(A)}{\mm(A)}\geq 2\sup_{t>0}\Big\{\frac{1-e^{-\lambda_0 t}}{J_K(t)}\Big\}, 
\end{equation}
for any Borel subset $A\subset X$ with $\mm(A)<\infty$. The estimate \eqref{implicitlambda0Intro}  follows. 
\end{proof}

\subsection{From the implicit to explicit bounds (and sharpness in case $K>0$).\\ Proof of Corollary \ref{cor:explicitbounds}}\label{sez: da impliciti a espliciti}

In this section we show how to derive explicit bounds for  $\lambda_1$ (resp. $\lambda_{0}$) in term of the Cheeger constant $h(X)$, starting from \eqref{implicitlambda1Intro} (resp. \eqref{implicitlambda0Intro}). We also show that \eqref{implicitlambda1Intro} is sharp, since equality is achieved on the Gaussian space.
\\First of all,  the expression of the function $J_K$  defined in \eqref{eq:defJKt} can be explicitly computed as:

\begin{equation}\label{eq:ExplJK}
J_K(t)=\begin{cases}\sqrt{\frac{2}{\pi K}}\arctan\Big(\sqrt{e^{2Kt}-1}\Big)  \ \ &\textrm{if} \ \  K>0,\\
\frac{2}{\sqrt{\pi}}\sqrt{t} \ \ &\textrm{if} \ \ K=0,\\
\sqrt{-\frac{2}{\pi K}}\arctanh{\Big(\sqrt{1-e^{2Kt}}\Big)} \ \ &\textrm{if} \ \ K<0.\end{cases} \qquad \forall t>0
\end{equation}

\subsubsection*{ \textbf{Case $K=0$}}
When $K=0$, the estimate \eqref{implicitlambda1Intro} combined with \eqref{eq:ExplJK} gives  

\begin{equation}
h(X)\geq \frac{\sqrt{\pi}}{2}\sup_{t>0}\frac{1-e^{-\lambda_1t}}{\sqrt{t}}=\frac{\sqrt{\pi\lambda_1}}{2}\sup_{T>0}\frac{1-e^{-T}}{\sqrt{T}},
\end{equation}
where we set $T=\lambda_{1} t$ in the last identity.

Let $W_{-1}:[-1/e,0)\rightarrow (-\infty,-1]$ be  the lower branch of the Lambert function, i.e. the inverse of the function $x\mapsto xe^x$ in the interval $(-\infty,-1]$. An easy computation yields
\begin{equation}
M:=\sup_{T>0}\frac{1-e^{-T}}{\sqrt{T}}=\frac{\sqrt{-4W_{-1}\Big(-\frac{1}{2\sqrt{e}}\Big)-2}}{2W_{-1}\Big(-\frac{1}{2\sqrt{e}}\Big)}, \ \  \mathrm{achieved \ at} \ T=-W_{-1}\Big(-\frac{1}{2\sqrt{e}}\Big)-\frac{1}{2}.
\end{equation} 
A good lower estimate of $M$ is given by $2/\pi$. Using this bound, we obtain
$$\lambda_1< \pi h^2(X).$$

\subsubsection*{ \textbf{Case $K>0$}} 
We start with the following
\begin{lemma}\label{lemma: f_1 crescente}
Let $f_1:(0,\infty)\rightarrow (0,\infty)$ be defined as
\begin{equation}
f_1(x):=\frac{\sqrt{x}}{\arctan\Big(\sqrt{e^{Tx}-1}\Big)}, 
\end{equation}
where $T>0$ is a fixed number.
Then $f_1$ is an increasing function and $f_1(x)\geq \frac{1}{\sqrt{T}}.$
\end{lemma}
\begin{proof}
The function $f_1$ is differentiable and the derivative of $f_1$ is non-negative if and only if 
\begin{equation*}
\sqrt{e^{Tx}-1}\arctan\big(\sqrt{e^{Tx}-1}\big)-Tx\geq 0, \ \ \ x>0.
\end{equation*}
We put $y:=\sqrt{e^{Tx}-1}$ so that we have to prove 
\begin{equation}\label{dis: g_1}
y\arctan(y)-\log(y^2+1)\geq 0, \ \ \ y>0.
\end{equation}
Called $g_1(y)$ the function $g_1(y):=y\arctan(y)-\log(y^2+1)$, we have that $g_1(0)=0$ and 
$$g'_1(y)=\arctan(y)-\frac{y}{1+y^2}\geq 0,$$
so that the inequality \eqref{dis: g_1} is proved and $f_1$ is increasing for any $T>0$.
The proof is finished since
$$\lim_{x\downarrow 0}f_1(x)=\frac{1}{\sqrt{T}}.$$
\end{proof}
Rewriting the estimate \eqref{implicitlambda1Intro} using \eqref{eq:ExplJK} in case $K>0$, we obtain
\begin{align}
\sqrt{\frac{2}{\pi}}h(X)&\geq \sqrt{K}\sup_{t>0}\frac{1-e^{-\lambda_1t}}{\arctan\Big(\sqrt{e^{2Kt}-1}\Big)} \nonumber\\
&=\sqrt{\lambda_1}\sup_{T>0}\frac{\sqrt{\frac{K}{\lambda_1}}}{\arctan\bigg(\sqrt{e^{2\frac{K}{\lambda_1}T}-1}\bigg)}\Big(1-e^{-T}\Big).\label{ug: riscrittura bound K>0}
\end{align}

Thanks to the Lemma \ref{lemma: f_1 crescente} it is clear that we can always obtain the same lower bound of the case $K=0$ (as expected), but this can be improved as soon as we have a positive lower bound of the quotient $K/\lambda_1$. 
Indeed, let us suppose $K/\lambda_1\geq c>0.$  Then, observing that 
$$\sup_{T>0} \frac{1-e^{-T}}{\arctan(\sqrt{e^{2cT}-1})}\geq \lim_{T\to +\infty} \frac{1-e^{-T}}{\arctan(\sqrt{e^{2cT}-1})}=\frac{2}{\pi},$$
from \eqref{ug: riscrittura bound K>0}, we obtain 
\begin{equation}\label{dis: bound K>0 e K/lambda>c}
\sqrt{\frac{2}{c\pi}}h(X)\geq \sqrt{\lambda_1}\sup_{T>0}\frac{1-e^{-T}}{\arctan(\sqrt{e^{2cT}-1})}\geq  \frac{2}{\pi} \sqrt{\lambda_1}.
\end{equation}

When $X=\mathbb{R}^d$ endowed with the Euclidean distance $\mathsf{d}(x,y)=|x-y|$  and the Gaussian measure $(2\pi)^{-d/2}e^{-|x|^2/2}d\mathcal{L}^d$, $1\leq d\in \mathbb{N}$, we have that $h(X)=\sqrt{\frac{2}{\pi}}$, $K=1$ and $\lambda_1=1$ (see \cite[Section 4.1]{BGL2}). Thus, we can take $c=1$ and the equality in \eqref{dis: bound K>0 e K/lambda>c} is achieved, making sharp the lower bound.

\subsubsection*{ \textbf{Case $K<0$}}
We begin by noticing that 
\begin{equation}\label{eq:JKKNeglog}
J_K(t)=\sqrt{-\frac{2}{\pi K}}\arctanh{\Big(\sqrt{1-e^{2Kt}}\Big)}=\sqrt{-\frac{2}{\pi K}}\log\Big(e^{-Kt}+\sqrt{e^{-2Kt}-1}\Big).
\end{equation}

The following lemma holds:

\begin{lemma}\label{lemma: f_2 decrescente}
Let $f_2:(0,\infty)\rightarrow (0,\infty)$ be defined as
\begin{equation}
f_2(x):=\frac{\sqrt{x}}{\log\big(e^{Tx}+\sqrt{e^{2Tx}-1}\big)}, 
\end{equation}
where $T>0$ is a fixed number.
Then $f_2$ is a decreasing function.
\end{lemma}
\begin{proof}
A direct computation shows that the derivative of $f_2$ is non-positive if and only if 
\begin{equation*}
\sqrt{e^{2Tx}-1} \; \log\Big(e^{Tx}+\sqrt{e^{2Tx}-1}\Big)\leq 2Txe^{Tx},\quad \text{for all }   x>0,
\end{equation*}
which is equivalent to 
\begin{equation}\label{dis: versione in x f_2}
\sqrt{1-e^{-2Tx}} \; \log\Big(1+\sqrt{1-e^{-2Tx}}\Big)\leq \Big(2-\sqrt{1-e^{-2Tx}}\Big)Tx, \quad \text{for all }x>0.
\end{equation}
We put $y:=\sqrt{1-e^{-2Tx}}$, and we write \eqref{dis: versione in x f_2} as
\begin{equation*}
y\log(1+y)+\frac{1}{2}(2-y)\log(1-y^2)\leq 0, \quad \text{for all }0<y<1,
\end{equation*}
which in turn is equivalent to 
\begin{equation}\label{eq:lastExprlemma2}
\left(1+\frac{y}{2} \right)\log(1+y)+\left(1-\frac{y}{2}\right)\log(1-y)\leq 0, \quad \text{for all }0<y<1.
\end{equation}
Now define $g_{2}:(0,1)\to \R$ as $g_2(y):=(1+\frac{y}{2})\log(1+y)+(1-\frac{y}{2})\log(1-y)$ and observe that  $g_2$ is concave with $g_2(0)=0$, $g'_2(0)=0$. Thus $g_{2}$ is non-positive on $(0,1)$ and the inequality \eqref{eq:lastExprlemma2} is proved.
\end{proof}
The combination of  \eqref{implicitlambda1Intro},  \eqref{eq:ExplJK} and \eqref{eq:JKKNeglog} implies that if $(X,\mathsf{d},\mm)$ is an $\mathsf{RCD}(K,\infty)$ space with $K<0$ and $\mm(X)=1$ then
\begin{equation}\label{eq:HlamndaKneg}
h(X)\geq \sqrt{-\frac{\pi K}{2}}\sup_{t>0}\frac{1-e^{-\lambda_1t}}{\log\Big(e^{-Kt}+\sqrt{e^{-2Kt}-1}\Big)}.
\end{equation}
We make two different choices: 
\begin{itemize}
\item When $\lambda_1\leq -K$, we choose $t=-\frac{1}{K}$ in \eqref{eq:HlamndaKneg} so that 
\begin{equation}\label{eq:la1h2X1}
h(X)\geq \sqrt{-\frac{\pi K}{2}}\frac{1-e^{\frac{\lambda_1}{K}}}{\log\Big(e+\sqrt{e^2-1}\Big)} \geq \lambda_1\sqrt{-\frac{\pi}{2K}}\frac{1-\frac{1}{e}}{\log\Big(e+\sqrt{e^2-1}\Big)},
\end{equation}
where we used the inequality 
$$1-e^{-x}\geq \left(1-\frac{1}{e} \right)x, \quad \text{for all } 0\leq x\leq 1.$$

\item When $\lambda_1>-K$, we choose $t=\frac{1}{\lambda_1}$ in \eqref{eq:HlamndaKneg} so that 
\begin{equation*}
h(X)\geq \sqrt{\frac{\pi}{2}}\sqrt{\lambda_1}\left(1-\frac{1}{e}\right)\frac{\sqrt{-\frac{K}{\lambda_1}}}{\log\bigg(e^{-\frac{K}{\lambda_1}}+\sqrt{e^{-2\frac{K}{\lambda_1}}-1}\bigg)}. 
\end{equation*}
Applying now Lemma \ref{lemma: f_2 decrescente}, we obtain
\begin{equation}\label{eq:la1h2X}
\lambda_1\leq \frac{2\Big(\log\big(e+\sqrt{e^{2}-1}\big)\Big)^2}{\pi\Big(1-\frac{1}{e}\Big)^2}h(X)^2.
\end{equation}
\end{itemize}
The combination of \eqref{eq:la1h2X1} and \eqref{eq:la1h2X} gives that, if $(X,\mathsf{d},\mm)$ is an $\mathsf{RCD}(K,\infty)$ space with $K<0$ and $\mm(X)=1$
\begin{multline}\label{eq:65}
\lambda_1\leq \max\bigg\{\sqrt{-K}\frac{\sqrt{2}\log\big(e+\sqrt{e^{2}-1}\big)}{\sqrt{\pi}(1-\frac{1}{e})}h(X),\frac{2\Big(\log\big(e+\sqrt{e^{2}-1}\big)\Big)^2}{\pi\Big(1-\frac{1}{e}\Big)^2}h(X)^2\bigg\}\\
<\max\left\{\frac{21}{10} \sqrt{-K}h(X),\frac{22}{5}h(X)^2 \right\}.
\end{multline}

In case $(X,\mathsf{d},\mm)$ is an $\mathsf{RCD}(K,\infty)$ space with $K<0$ and $\mm(X)=\infty$ then, using \eqref{implicitlambda0Intro} instead of \eqref{implicitlambda1Intro}, the estimates \eqref{eq:HlamndaKneg} and \eqref{eq:65}   hold with $\lambda_{1}$ replaced by $\lambda_{0}$ and $h(X)$ replaced by $h(X)/2$. Thus, in case $\mm(X)=\infty$, we obtain:
\begin{multline}
\lambda_0\leq \max\bigg\{\sqrt{-K}\frac{\log\big(e+\sqrt{e^{2}-1}\big)}{\sqrt{2\pi}(1-\frac{1}{e})}h(X),\frac{\Big(\log\big(e+\sqrt{e^{2}-1}\big)\Big)^2}{2\pi\Big(1-\frac{1}{e}\Big)^2}h(X)^2\bigg\}\\
<\max\left\{\frac{21}{20} \sqrt{-K}h(X),\frac{11}{10}h(X)^2 \right\}.
\end{multline}

\hfill$\Box$

\begin{remark}
Another bound, similar to the one obtained  in the case $K>0$, can be achieved in the presence of a lower bound for $K/\lambda_1$, if $\mm(X)=1$ (resp. a lower bound for $K/\lambda_0$, if $\mm(X)=\infty$). To see this, let us suppose $K/\lambda_1\geq -c, \ c>0$ (resp. $K/\lambda_0\geq -c$). Then, using  \eqref{implicitlambda1Intro} (resp. \eqref{implicitlambda0Intro}), \eqref{eq:ExplJK}  and Lemma \ref{lemma: f_2 decrescente}, we have that (resp. the left-hand side can be improved to $h(X)/\sqrt{2 \pi}$)
\begin{multline}
\sqrt{\frac{2}{\pi}}h(X)\geq \sqrt{\lambda_1}\sup_{T>0}\frac{\sqrt{-\frac{K}{\lambda_1}}}{\log\bigg(e^{-\frac{K}{\lambda_1}T}+\sqrt{e^{-2\frac{K}{\lambda_1}T}-1}\bigg)}\Big(1-e^{-T}\Big)\\
\geq \sqrt{c\lambda_1}\sup_{T>0}\frac{1-e^{-T}}{\log\big(e^{cT}+\sqrt{e^{2cT}-1}\big)}.
\end{multline}
\end{remark}

\section{Appendix A: Cheeger's inequality in general metric measure spaces}
The Buser-type inequalities of Theorem  \ref{thm:MainImpl} and Corollary  \ref {cor:explicitbounds}  give an upper bound on $\lambda_{1}$ (resp. on $\lambda_{0}$, in case $\mm(X)=\infty$) in  terms of the Cheeger constant $h(X)$. It is natural to ask if also a reverse inequality holds, namely if it possible to give a lower bound on  $\lambda_{1}$  (resp. on $\lambda_{0}$, in case $\mm(X)=\infty$) in terms of $h(X)$. The answer is affirmative in the higher generality of metric measure spaces with a non-negative locally bounded measure \emph{without curvature conditions}, see  Theorem \ref{thm:Cheeger}  below.
This generalizes to the metric measure setting a celebrated result by Cheeger  \cite{ChIn}, known as Cheeger's inequality. In contrast to the previous section, here we do not assume the separability of the space.
 
 A key tool in the proof of Cheeger's inequality is the co-area formula; more precisely, in the arguments it is enough to have an inequality in the co-area formula. For the reader's convenience, we give below the statement and a self-contained proof.

\begin{proposition}[Coarea inequality]\label{thm:coarea}
Let $(X,\mathsf{d})$ be a complete metric space and let $\mm$ be a non-negative Borel measure finite on bounded subsets.
\\Let $u\in {\mathsf {Lip}}_{b}(X)$, $u:X\to [0,\infty)$ and set $M=\sup_{X} u$. 
Then for $\mathcal{L}^{1}$-a.e. $t>0$ the set $\{u> t \}$ has finite perimeter and 
\begin{equation}\label{eq:coarea}
\int_{0}^{M} \Per(\{u>t\}) \, d t \leq \int_{X} |\nabla u| \, d\mm.
\end{equation}
\end{proposition}

\begin{proof}
The proof is quite standard, but since we did not find it in the literature stated at this level of generality (tipically one assumes some extra condition like measure doubling and gets a stronger statement, namely equality in the co-area formula; see for instance \cite{MirandaJr}) we add it for the reader's convenience.
\\Let $E_{t}:= \{u> t \}$ and set $V(t):=\int_{E_{t}} |\nabla u|\, d\mm$. The function $t\mapsto V(t)$ is non-increasing and bounded, thus differentiable for $\mathcal{L}^{1}$-a.e. $t>0$.
\\Since $\int_{X} u \, d\mm<\infty$, we also have that $\mm(\{u=t\})=0$ for  $\mathcal{L}^{1}$-a.e. $t>0$.

Fix $t>0$  a differentiability point for $V$ for which $\mm(\{u=t\})=0$,  and define $\psi:(0,t) \times (0,\infty) \to [0,1]$ as
\begin{equation}\label{eq:defpsi}
\psi(h,s):= 
\begin{cases} 0 &\quad \text{for }  s\leq  t-h \\
\frac{1}{h}(s-t)+1 &\quad \text{for }  t-h <s\leq t  \\
1 & \quad \text{for } s> t.
\end{cases}
\end{equation}
For $h>0$ define $u_{h}(x)=\psi(h, u(x))$ and observe that the sequence $(u_{h})_{h}\subset {\mathsf {Lip}}_{b}(X)$. 
\\We first claim that
\begin{equation}\label{eq:uhtoEt}
u_{h}\to \chi_{E_{t}}\quad \text{in } L^{1}(X,\mm) \quad \text{as } h\downarrow 0. 
\end{equation}
Indeed
\begin{align*}
\int_{X} |u_{h}-\chi_{E_{t}}|\, d\mm&=  \int_{\{ t-h<u\leq t\}} \psi(h,u)  \, d\mm \\
&\leq  \mm\left(\left\{ t-h<u\leq t \right\} \right) \to \mm(\{u=t\})=0 \quad \text{ as } h\downarrow 0,
\end{align*}
by Dominated Convergence Theorem, since by assumption $u$ has bounded support, $\mm$ is finite on bounded sets and $\chi_{\{ t-h<u\leq t\}} \to \chi_{\{u=t\}}$ pointwise as   $h\downarrow 0$.
\\In order to prove that $E_{t}$ is a set of finite perimeter it is then sufficient to show that  $\limsup_{h\downarrow 0} \int_{X} |\nabla u_{h}|\, d\mm<\infty$. To this aim observe that
$$
 \int_{X} |\nabla u_{h}|\, d\mm =  \frac{1} {h} \int_{{\{ t-h<u\leq t\}}} |\nabla u|  \, d\mm= \frac{V(t-h)-V(t)}{h}.
$$
Since by assumption $t>0$ is a differentiability point for $V$, we obtain that $E_{t}$ is a finite perimeter set satisfying
\begin{equation}\label{eq:V't}
\Per(E_{t}) \leq \lim_{h\downarrow 0}  \int_{X} |\nabla u_{h}|\, d\mm=- V'(t).
\end{equation}
Using that \eqref{eq:V't} holds for  $\mathcal{L}^{1}$-a.e. $t>0$ and that $V$ is non-increasing, we get
\begin{equation}\label{eq:coarea2}
\int_{0}^{M} \Per(E_{t}) \, d t \leq - \int_{0}^{M} V'(t) \, d t   \leq V(0)-V(M)=\int_{X} |\nabla u| \, d\mm.
\end{equation}
 
\end{proof}

\begin{theorem}[Cheeger's Inequality in metric measure spaces]\label{thm:Cheeger}
Let $(X,\mathsf{d})$ be a complete metric space and let $\mm$ be a non-negative Borel measure finite on bounded subsets.
\begin{enumerate}
\item If $\mm(X)<\infty$ then
\begin{equation}\label{eq:Cheeger1}
\lambda_{1}\geq \frac{1}{4} h(X)^{2}.
\end{equation}
\item  If $\mm(X)=\infty$ then
\begin{equation}\label{eq:Cheeger0}
\lambda_{0}\geq \frac{1}{4} h(X)^{2}.
\end{equation}
\end{enumerate}
\end{theorem}

As proved by Buser \cite{Buser78}, the constant $1/4$ in \eqref{eq:Cheeger1} is  optimal  in the following sense: for any $h > 0$ and $\varepsilon>0$, there exists a closed (i.e. compact without boundary) two-dimensional Riemannian manifold $(M,g)$ with $h(M)=h$ and such that $\lambda_{1}\leq  \frac 1 4 h(M)^{2}+\varepsilon$.

\begin{proof}
We give a proof of \eqref{eq:Cheeger1}, the arguments for showing  \eqref{eq:Cheeger0} being analogous (and even simpler).
\\By the very definition of $\lambda_{1}$ as in \eqref{eq:defla1Intro}, for every $\varepsilon>0$ there exists $f\in \mathsf{Lip}_{b}(X)$ with $\int_X f\, d\mm=0$,  $f\not\equiv 0$ such that
\begin{equation}\label{eq:PfCh1}
\lambda_1\geq \frac{\int_X |\nabla f|^2 \, d\mm}{\int_X f^2 \, d\mm} -\varepsilon.
\end{equation}
Let $m$ be any median of the function $f$ and set $f^+:=\max \{f-m,0\}$, $f^-:=-\min \{f-m,0\}$. 
Applying the co-area inequality \eqref{eq:coarea} to $u=(f^+)^{2}$ (respectively $(f^-)^2$) and recalling the definition of Cheeger's constant $h(X)$ as in  \eqref{eq:defChConst}, we obtain
\begin{align}\label{eq:PfCh3}
&\int_X |\nabla (f^+)^{2}| \, d\mm + \int_X |\nabla (f^-)^{2}| \, d\mm  \\
&\geq\int_{0}^{\sup \{(f^+)^{2}\}}  \Per(\{(f^+)^{2}>t\}) \, d t +\int_{0}^{\sup \{(f^-)^{2}\}}  \Per(\{(f^-)^{2}>t\}) \, d t  \nonumber \\ 
&\geq h(X) \int_{0}^{\sup \{(f^+)^{2}\}}   \mm(\{(f^+)^{2}>t\}) \, d t +h(X) \int_{0}^{\sup \{(f^-)^{2}\}}   \mm(\{(f^-)^{2}>t\}) \, d t  \nonumber \\
&=h(X)\int_X  (f^+)^{2}  d\mm+h(X) \int_X (f^-)^2  d\mm = h(X)  \int_{X} |f-m|^{2} \, d\mm. \nonumber
\end{align}
Since 
$$|\nabla g^{2}|\leq 2 |g| \, |\nabla g|,$$
and 
$$|\nabla f^+|\leq |\nabla f|, \ \  |\nabla f^-|\leq |\nabla f|,$$
we can apply the Cauchy-Schwarz inequality and get
\begin{equation}\label{eq:PfCh4}
2\bigg( \int_X |\nabla f|^2 \, d\mm \bigg)^{\frac{1}{2}}\bigg( \int_X |f-m|^2 \, d\mm \bigg)^{\frac{1}{2}}\geq \int_X |\nabla (f^+)^{2}| \, d\mm + \int_X |\nabla (f^-)^{2}| \, d\mm,
\end{equation}
where we have used that  $|f^+| + |f^-|=|f-m|$.
It follows from \eqref{eq:PfCh3} and \eqref{eq:PfCh4} that for every median $m$ of $f$ it holds
\begin{equation}
\frac{\int_X |\nabla f|^2 \, d\mm}{\int_X |f-m|^2 \, d\mm}\geq \frac{h(X)^2}{4}.
\end{equation}
Finally, since $\int_X f d\mm=0$ and the mean minimises $\R\ni c\mapsto \int_X |f-c|^2 d\mm$, we have
$$\frac{\int_X |\nabla f|^2 \, d\mm}{\int_X |f|^2 \, d\mm}\geq \frac{\int_X |\nabla f|^2 \, d\mm}{\int_X |f-m|^2 \, d\mm} $$ 
and we can conclude thanks to \eqref{eq:PfCh1} and the fact that $\varepsilon>0$ is arbitrary.
\end{proof}

\end{document}